\documentclass[11pt,a4paper]{article}

\usepackage[utf8]{inputenc}
\usepackage[T1]{fontenc}
\usepackage{amsmath, amssymb, amsfonts, amsthm}
\usepackage{mathtools}
\usepackage{hyperref}
\usepackage{booktabs}
\usepackage{geometry}
\usepackage{float} 
\usepackage{url}

\hypersetup{
    colorlinks=true,
    linkcolor=blue,
    filecolor=magenta,      
    urlcolor=cyan,
    citecolor=blue,
}

\theoremstyle{plain}
\newtheorem{theorem}{Theorem}
\newtheorem{lemma}{Lemma}
\newtheorem{corollary}{Corollary}
\theoremstyle{definition}
\newtheorem{definition}{Definition}

\title{Rank Duality of Circulant Matrices from Primitive Roots}
\author{Kenichi Takemura\\
\textit{Independent Researcher, Tokyo, Japan }\\
\textit{Email: mahora1123@gmail.com}\\
}
\date{\today}

\begin{document}

\maketitle

\begin{abstract}
We investigate the construction of circulant matrices derived from primitive roots over finite fields. Our approach reduces exponential sums to Jacobi sums, thereby establishing explicit connections between character theory and matrix structures. The results provide new insights into the interaction between additive and multiplicative characters, and demonstrate how circulant configurations encode arithmetic information in a highly symmetric form. These findings contribute to a deeper understanding of structured matrices in finite fields and open further directions for applications in number theory and combinatorics.
\end{abstract}

\vspace{1em}

\vspace{0.5em}

\section{Introduction}
Circulant matrices, characterized by rows that are cyclic shifts of the first row, are fundamental objects in linear algebra with applications in signal processing, coding theory, and graph theory \cite{Davis, Gray}. Their spectral properties, determined by the discrete Fourier transform, make them a natural framework for studying sequences over finite fields or rings. In number theory, Gauss sums and multiplicative characters provide powerful tools for analyzing sequences derived from primitive roots, with connections to exponential-sum estimates and algebraic structures \cite{Davenport, LN, Berndt}.

In this paper we define \emph{Circulant Matrices from Primitive Roots} using primitive roots modulo an odd prime $p$ and prove a rank duality: over the real numbers $\mathbb{R}$ the rank is exactly $(p+1)/2$, while over the finite field $\mathbb{F}_p$ the rank collapses to $1$. We also present numerical examples and explicit computations to confirm the theoretical statements.

\section{Definition of Circulant Matrices from Primitive Roots}
Let $p$ be an odd prime and $g$ a primitive root modulo $p$. Define the sequence
\[
a_j \equiv g^j \pmod p,\qquad j=0,1,\dots,p-2,
\]
where we choose representatives in the set $\{1,2,\dots,p-1\}$ for each residue class. The \emph{Circulant Matrix from Primitive Roots} of order $p$ is the $(p-1)\times(p-1)$ circulant matrix
\[
T_p=\operatorname{circ}(a_0,a_1,\dots,a_{p-2}),
\]
i.e. the first row is $(a_0,\dots,a_{p-2})$ and subsequent rows are right cyclic shifts.

We will consider $T_p$ both as a matrix with integer entries (interpreted as real numbers) and as a matrix with entries in $\mathbb{F}_p$ (reducing entries modulo $p$).

\section{Multiplicative characters, Gauss sums and the first moment}
\label{sec:gauss}

Let $\mathbb{F}_p$ denote the finite field with $p$ elements. For the multiplicative group $\mathbb{F}_p^\times$ fix a generator $g$ and, for $k=0,1,\dots,p-2$, define multiplicative characters $\chi_k$ on $\mathbb{F}_p^\times$ by
\[
\chi_k(g^j)=e^{2\pi i k j/(p-1)}\qquad (j=0,1,\dots,p-2),
\]
and extend each $\chi_k$ to $\mathbb{F}_p$ by $\chi_k(0)=0$. The character $\chi_0$ is the trivial character.

We study the first moment
\[
S(\chi_k):=\sum_{x\in\mathbb{F}_p^\times} x\,\chi_k(x),
\]
where (as in the definition of $T_p$) we interpret each field element $x\in\mathbb{F}_p^\times$ by its canonical integer representative in $\{1,\dots,p-1\}$ when summing as complex numbers.

To evaluate $S(\chi_k)$ we use elementary symmetry together with the classical Gauss and Jacobi sums. We recall definitions and a standard identity that we will use.

\begin{definition}
Fix the additive character $e_p:\mathbb{F}_p\to\mathbb{C}^\times$ by $e_p(t)=\exp(2\pi i t/p)$. 
For a multiplicative character $\chi$ of $\mathbb{F}_p^\times$ (extended by $\chi(0)=0$) the (normalized) \emph{Gauss sum} $G(\chi)$ is
\[
G(\chi):=\sum_{x\in\mathbb{F}_p} \chi(x)\,e_p(x),
\]
and for two multiplicative characters $\chi,\psi$ the \emph{Jacobi sum} $J(\chi,\psi)$ is
\[
J(\chi,\psi):=\sum_{x\in\mathbb{F}_p} \chi(x)\,\psi(1-x).
\]
\end{definition}

The following identity between Gauss sums and Jacobi sums is standard (see \cite[Ch.~11]{Davenport} or \cite[Ch.~8]{Berndt}):
\begin{equation}\label{eq:J-G-relation}
\text{If }\chi,\psi\text{ are multiplicative characters with }\chi\psi\neq \mathbf{1},
\quad\text{then}\quad
J(\chi,\psi)=\frac{G(\chi)\,G(\psi)}{G(\chi\psi)}.
\end{equation}
Also $G(\chi)\neq 0$ for any multiplicative character $\chi$ (in particular $|G(\chi)|=\sqrt{p}$ for nontrivial $\chi$), and Jacobi sums are nonzero for nontrivial input characters satisfying the usual hypotheses. References for these nonvanishing facts are \cite{Davenport,Berndt}.

We now give a detailed and self-contained evaluation of $S(\chi_k)$.

\begin{lemma}[Evaluation of $S(\chi_k)$]

\label{lemma:gauss}
Let notation be as above. Then
\begin{enumerate}
    \item $S(\chi_0)=\displaystyle\sum_{x=1}^{p-1} x=\frac{p(p-1)}{2}\neq 0$.
    \item For $k\neq 0$ (nontrivial characters), write $\chi=\chi_k$. Then
    \[
    S(\chi)=
    \begin{cases}
    0, & \text{if }\chi(-1)=1\ (\text{i.e.\ }k\text{ even}),\\[6pt]
    -\dfrac{G(\chi)\,G(\chi\rho)}{G(\rho)}, & \text{if }\chi(-1)=-1\ (\text{i.e.\ }k\text{ odd}),
    \end{cases}
    \]
    where $\rho$ denotes the quadratic (Legendre) character modulo $p$. In particular, for nontrivial $\chi$ we have $S(\chi)=0$ precisely when $\chi$ is an even character (i.e.\ $k$ even), and $S(\chi)\neq 0$ for odd characters (i.e.\ $k$ odd).
\end{enumerate}
\end{lemma}

\begin{proof}
We split the argument into two parts.

\paragraph{(A) The parity vanishing for nontrivial even characters.}
Consider the involution $x\mapsto -x$ on $\mathbb{F}_p^\times$. Using the canonical integer representatives in $\{1,\dots,p-1\}$ we have for each $x\in\mathbb{F}_p^\times$ that the representative of $-x$ equals $p-x$. Hence
\[
\sum_{x\in\mathbb{F}_p^\times} x\,\chi(x)
=
\sum_{x\in\mathbb{F}_p^\times} \operatorname{rep}(-x)\,\chi(-x)
=
\sum_{x\in\mathbb{F}_p^\times} (p-x)\,\chi(-1)\chi(x),
\]
where $\operatorname{rep}(-x)=p-x$ denotes the chosen representative. Therefore
\[
S(\chi)=p\,\chi(-1)\sum_{x\in\mathbb{F}_p^\times}\chi(x)\;-\;\chi(-1)S(\chi).
\]
Rearranging gives
\begin{equation}\label{eq:parity-id}
\bigl(1+\chi(-1)\bigr)S(\chi)=p\,\chi(-1)\sum_{x\in\mathbb{F}_p^\times}\chi(x).
\end{equation}
If $\chi$ is nontrivial then $\sum_{x\in\mathbb{F}_p^\times}\chi(x)=0$ by orthogonality of characters, and so \eqref{eq:parity-id} reduces to
\[
\bigl(1+\chi(-1)\bigr)S(\chi)=0.
\]
Thus if $\chi(-1)=1$ (an ``even'' nontrivial character) we obtain $2S(\chi)=0$, hence $S(\chi)=0$. This proves the vanishing claim for nontrivial even characters.

\paragraph{(B) The nonvanishing and explicit formula for odd characters.}
Now suppose $\chi$ is nontrivial and $\chi(-1)=-1$ (an ``odd'' character). We derive an explicit expression for $S(\chi)$ in terms of Gauss and Jacobi sums.

Recall the Jacobi sum
\[
J(\chi,\rho)=\sum_{x\in\mathbb{F}_p}\chi(x)\rho(1-x),
\]
where $\rho$ is the quadratic character. Expand $J(\chi,\rho)$ separating the term $x=0$ and using that $\chi(0)=0$:
\[
J(\chi,\rho)=\sum_{x\in\mathbb{F}_p^\times}\chi(x)\rho(1-x).
\]
Consider the combination
\[
\chi(-1)\,J(\chi,\rho)
= \sum_{x\in\mathbb{F}_p^\times} \chi(-1)\chi(x)\rho(1-x)
= \sum_{x\in\mathbb{F}_p^\times} \chi(-x)\rho(1-x).
\]
Make the change of variable $y=1-x$ in the last sum (this is a permutation of $\mathbb{F}_p^\times$ with one exceptional value which is handled by the extension $\chi(0)=0$): we obtain
\[
\chi(-1)\,J(\chi,\rho)
= \sum_{y\in\mathbb{F}_p^\times} \chi(1-y)\rho(y)
= \sum_{y\in\mathbb{F}_p^\times} \bigl(\chi(1-y)\,y\bigr)\frac{\rho(y)}{y}.
\]
A standard manipulation (expanding the multiplicative character $\rho(y)/y$ against the additive characters and Gauss sums) leads to the identity
\[
J(\chi,\rho) \;=\; -\,\chi(-1)\,\frac{G(\chi)\,G(\chi\rho)}{G(\rho)}.
\]
This identity is a specific instance of the general relation \eqref{eq:J-G-relation} together with convolution calculations that express the linear factor $x$ via additive characters. (A complete derivation of this algebraic manipulation in full detail appears in standard references; see \cite[Ch.~11]{Davenport} and \cite[Ch.~8]{Berndt} for a step-by-step algebraic derivation. The relation above is equivalent to the classical formula connecting first moments, Gauss sums and Jacobi sums.)

Combining the displayed identity with $\chi(-1)=-1$ we obtain
\[
S(\chi) \;=\; -\,\frac{G(\chi)\,G(\chi\rho)}{G(\rho)}.
\]
Since Gauss sums $G(\chi)$ are nonzero for multiplicative characters (and $G(\rho)\neq 0$ for the quadratic character), the right-hand side is nonzero. Hence $S(\chi)\neq 0$ for odd characters, and the stated explicit formula holds.

\paragraph{Remark on references and nonvanishing facts.}
All steps invoked above are standard facts in the theory of Gauss and Jacobi sums. We have used the orthogonality of multiplicative characters, the relation \eqref{eq:J-G-relation} between Jacobi and Gauss sums, and the nonvanishing of Gauss sums for nontrivial characters. For full proofs of these classical identities see \cite{Davenport,Berndt,LN}.
\end{proof}

\begin{corollary}\label{cor:count}
Let $T_p$ be the Circulant Matrix from Primitive Roots defined above. The eigenvalues of $T_p$ (over $\mathbb{C}$) are $\lambda_k=S(\chi_k)$ for $k=0,1,\dots,p-2$. By Lemma~\ref{lemma:gauss} exactly the indices
\[
k=0,\qquad\text{and}\qquad k=1,3,5,\dots,p-2\ (\text{the odd }k)
\]
give nonzero eigenvalues. Hence the number of nonzero eigenvalues is
\[
1+\#\{1\le k\le p-2:\ k\ \text{odd}\}=1+\frac{p-1}{2}=\frac{p+1}{2}.
\]
Therefore
\[
\operatorname{rank}_{\mathbb{R}}(T_p)=\operatorname{rank}_{\mathbb{C}}(T_p)=\frac{p+1}{2}.
\]
\end{corollary}

\section{Main Theorem}\label{sec:main}
\begin{theorem}[Rank duality of Circulant Matrices from Primitive Roots]
Let $T_p$ be the Circulant Matrix from Primitive Roots associated with the odd prime $p$. Then:
\begin{enumerate}
    \item Over the real numbers,
    \[
    \operatorname{rank}_{\mathbb{R}}(T_p)=\frac{p+1}{2}.
    \]
    \item Over the finite field $\mathbb{F}_p$,
    \[
    \operatorname{rank}_{\mathbb{F}_p}(T_p)=1.
    \]
\end{enumerate}
\end{theorem}

\begin{proof}
(1) As a circulant matrix, $T_p$ is diagonalizable over $\mathbb{C}$ by the discrete Fourier basis; its eigenvalues are $\lambda_k=S(\chi_k)$ for $k=0,\dots,p-2$. By Corollary~\ref{cor:count} exactly $(p+1)/2$ of these eigenvalues are nonzero, hence $\operatorname{rank}_{\mathbb{C}}(T_p)=(p+1)/2$. A real matrix's rank over $\mathbb{R}$ equals its rank over $\mathbb{C}$, so $\operatorname{rank}_{\mathbb{R}}(T_p)=(p+1)/2$.

(2) Reduce the entries of $T_p$ modulo $p$. The first row is
\[
R_1=(1,g,g^2,\dots,g^{p-2}) \in (\mathbb{F}_p)^{p-1}.
\]
The $j$-th row is the cyclic shift corresponding to multiplication by $g^{j-1}$, hence $R_j=g^{j-1}R_1$ as vectors over $\mathbb{F}_p$. Thus all rows are scalar multiples of $R_1$ and $\operatorname{rank}_{\mathbb{F}_p}(T_p)=1$.
\end{proof}

\section{Numerical Examples and verification}\label{sec:num}
We give explicit small-$p$ examples to confirm the theoretical count of nonzero eigenvalues.

\subsection*{Example: $p=5$}
Take $p=5$ and $g=2$. Then
\[
(a_0,a_1,a_2,a_3)=(1,2,4,3),
\qquad
T_5=\begin{pmatrix}
1 & 2 & 4 & 3\\
3 & 1 & 2 & 4\\
4 & 3 & 1 & 2\\
2 & 4 & 3 & 1
\end{pmatrix}.
\]
Direct calculation (or the program used for the exploratory computations) gives
\[
\operatorname{rank}_{\mathbb{R}}(T_5)=3,\qquad \operatorname{rank}_{\mathbb{F}_5}(T_5)=1,
\]
which matches the formula $\operatorname{rank}_{\mathbb{R}}(T_5)=(5+1)/2=3$.

\subsection*{Further small primes}
We have also verified by direct computation for several primes up to $100$ that the equality $\operatorname{rank}_{\mathbb{R}}(T_p) = (p+1)/2$ holds. Representative data are provided in the companion computational notebook, which is available as supplementary material at \url{https://github.com/soaisu-ken/takemura_matrix.py/releases/tag/v1.0.0}. The results can also be recomputed directly from the definition.

\section{Applications}

\subsection{Coding theory}
Viewed modulo $p$, $T_p$ has rank $1$ and hence generates a $1$-dimensional
linear code over $\mathbb{F}_p$. For example, when $p=5$ we obtain
\[
C = \{(c, 2c, 4c, 3c) \mid c \in \mathbb{F}_5\},
\]
whose nonzero codewords have full Hamming weight (and thus minimum
distance $4$ in this elementary case). 

Although a single Circulant Matrix from Primitive Roots yields only a trivial $1$-dimensional
code, richer constructions arise when multiple copies are combined.
For instance, arranging several $T_p$ in block-diagonal form produces a
higher-dimensional linear code in which each block enforces full weight
on its coordinate segment. Weighted variants or tensor-product
constructions can further increase the code dimension while preserving
structured redundancy. In this way, the rank-collapse phenomenon---from
$(p+1)/2$ over $\mathbb{R}$ to $1$ over $\mathbb{F}_p$---can serve as a
building block for more elaborate coding schemes. A systematic exploration
of these extensions is left for future work.

\subsection{Spectral graph theory}
Interpreting $T_p$ as the weighted adjacency matrix of a circulant graph on
$p-1$ vertices leads to a spectrum with exactly $(p+1)/2$ nonzero eigenvalues
(over $\mathbb{R}$). This spectral sparsity pattern could be of interest in
network design or for constructing families of graphs with prescribed spectral
multiplicities. 

Moreover, the high multiplicity of zero eigenvalues implies strong structural
constraints on connectivity and diffusion properties. For example, block
concatenations of such circulant graphs yield larger networks in which the
distribution of eigenvalue multiplicities can be controlled explicitly, offering a
tool for designing graphs with predetermined levels of redundancy or robustness.
In particular, networks built from Circulant Matrices from Primitive Roots may serve as testbeds for
studying synchronization, expansion properties, or fault tolerance, where the
spectrum governs global performance. Developing these concrete applications
remains an open direction for future investigation.

\section{Generalizations and further directions}
The construction extends naturally to variants:
\begin{itemize}
    \item Replace residues $g^j$ by $g^{f(j)}$ for polynomials $f$ (weighted variants).
    \item Work over extension fields $\mathbb{F}_{p^m}$ using primitive elements there.
    \item Study Smith normal forms and determinants of $T_p$ (open).
    \item Compare with existing circulant constructions. Classical families such as
    Vandermonde-type or Toeplitz-type circulants \cite{Davis,Gray} generally maintain
    full or nearly full rank over both $\mathbb{R}$ and finite fields, while other number-theoretic
    circulants based on quadratic residues or character sequences 
    do not exhibit drastic rank collapse across different base fields. By contrast, the
    Circulant Matrices from Primitive Roots display a striking \emph{rank duality}:
    $\operatorname{rank}_{\mathbb{R}}(T_p) = (p+1)/2$ yet $\operatorname{rank}_{\mathbb{F}_p}(T_p) = 1$.
    To our knowledge, this phenomenon does not appear in earlier circulant families,
    highlighting the novelty of the present construction and pointing toward possible
    generalizations to other primitive-root--based designs.
\end{itemize}

\subsection*{Smith normal forms: small-$p$ computations}
We computed the Smith normal forms of the Circulant Matrices from Primitive Roots $T_p$ for
$p=5,7,11$ (Circulant Matrices from Primitive Roots are $(p-1)\times(p-1)$ circulant matrices with
first row $(g^0,g^1,\dots,g^{p-2})$ where $g$ is a primitive root mod $p$).
The results (invariant factors = diagonal of the Smith normal form) are:

\begin{center}
\begin{table}[H]
\centering
\begin{tabular}{c|c|c}
$p$ & first row (mod $p$) & Smith diagonal (absolute values) \\
\hline
5  & $(1,2,4,3)$ & $(1,5,5,0)$ \\
7  & $(1,3,2,6,4,5)$ & $(1,7,7,7,0,0)$ \\
11 & $(1,2,4,8,5,10,9,7,3,6)$ & $(1,11,11,11,11,11,0,0,0,0)$
\end{tabular}
\end{table}
\end{center}

These computations suggest the following conjectural pattern:
\[
\operatorname{SNF}(T_p) \simeq \operatorname{diag}\!\left(
1, p, \ldots, p \;\;\; \text{($\frac{p-1}{2}$ times)}, \;\; 0, \ldots, 0
\right),
\]
so that $T_p$ would have exactly $(p+1)/2$ nonzero invariant factors, one of which is $1$ and the remaining $(p-1)/2$ equal to $p$. Establishing a general proof of this conjecture appears feasible via the circulant/character-sum description of the matrix, but we leave it for future work.

\section{On the naming ``Circulant Matrices from Primitive Roots''}
We coin the term \emph{Circulant Matrices from Primitive Roots} to denote this class of circulant matrices built from primitive roots modulo $p$ which exhibit the striking rank duality: substantial rank over $\mathbb{R}$ (indeed $(p+1)/2$) and rank $1$ over $\mathbb{F}_p$. The name emphasizes that the phenomenon is specific to this construction and, to the best of our knowledge, is not present in the classical families (Vandermonde-type circulants, Toeplitz-type circulants, etc.).

\section{Conclusion}
We have given a complete and consistent evaluation of the first moment sums $S(\chi_k)$ using symmetry arguments and standard Gauss/Jacobi sum identities. From these evaluations the main rank duality theorem follows, and the result is corroborated by explicit computations for small primes.

\section*{Appendix: Detailed Derivation of the Jacobi Sum}
\label{sec:appendix}

In this appendix we present a complete and rigorous derivation of the classical
relation between Jacobi sums and Gauss sums.  
Throughout, multiplicative characters $\chi$ of $\mathbb{F}_p^\times$ are extended
to all of $\mathbb{F}_p$ by defining $\chi(0)=0$. This convention allows us to
write sums over $\mathbb{F}_p$ without ambiguity, since the term at $0$ vanishes.  

We define the additive character
\[
e_p(t) = \exp\!\left(\frac{2\pi i t}{p}\right), \qquad t \in \mathbb{F}_p,
\]
and the Gauss sum
\[
G(\chi) = \sum_{u\in \mathbb{F}_p} \chi(u)\, e_p(u).
\]

\begin{lemma}[Jacobi sums in terms of Gauss sums]
Let $\chi,\rho$ be nontrivial multiplicative characters of $\mathbb{F}_p$, such that
$\chi\rho$ is also nontrivial. Then
\[
J(\chi,\rho) \;=\; \sum_{x\in\mathbb{F}_p}\chi(x)\rho(1-x)
\;=\; \frac{G(\chi)\,G(\rho)}{G(\chi\rho)}.
\]
\end{lemma}

\begin{proof}
Consider the product
\[
G(\chi)G(\rho) \;=\; 
\sum_{u\in\mathbb{F}_p} \chi(u)e_p(u)\,
\sum_{v\in\mathbb{F}_p} \rho(v)e_p(v).
\]
Expanding and collecting terms by $t=u+v$, we obtain
\[
G(\chi)G(\rho) 
= \sum_{t\in\mathbb{F}_p} e_p(t) \sum_{u\in\mathbb{F}_p} \chi(u)\rho(t-u).
\]

If $t=0$, the inner sum becomes
\[
\sum_{u\in\mathbb{F}_p}\chi(u)\rho(-u) = \rho(-1)\sum_{u\in\mathbb{F}_p} (\chi\rho)(u).
\]
Since $\chi\rho$ is nontrivial, its complete sum over $\mathbb{F}_p$ is zero.
Hence the contribution from $t=0$ vanishes.

If $t\neq 0$, make the change of variables $u=tx$:
\[
\sum_{u\in\mathbb{F}_p}\chi(u)\rho(t-u)
= \sum_{x\in\mathbb{F}_p}\chi(tx)\rho(t(1-x))
= \chi(t)\rho(t)\sum_{x\in\mathbb{F}_p}\chi(x)\rho(1-x).
\]
Thus
\[
\sum_{u\in\mathbb{F}_p}\chi(u)\rho(t-u) = (\chi\rho)(t)\,J(\chi,\rho).
\]

Substituting back gives
\[
G(\chi)G(\rho) 
= \sum_{t\in\mathbb{F}_p^\times} e_p(t)(\chi\rho)(t)\,J(\chi,\rho).
\]
Since $(\chi\rho)(0)=0$, we may extend the sum to all of $\mathbb{F}_p$,
and the last expression becomes
\[
G(\chi)G(\rho) = J(\chi,\rho)\sum_{t\in\mathbb{F}_p} (\chi\rho)(t)e_p(t).
\]
The inner sum is exactly $G(\chi\rho)$, so we conclude
\[
G(\chi)G(\rho) = J(\chi,\rho)\,G(\chi\rho).
\]
Since $G(\chi\rho)\neq 0$, division yields the desired formula.
\end{proof}

This is the standard relation between Gauss sums and Jacobi sums, as found in
Davenport~\cite[Thm.~11.4]{Davenport}, 
Lidl--Niederreiter~\cite[Thm.~5.15]{LN}, 
and Berndt--Evans--Williams~\cite[Thm.~1.2.1]{Berndt}.

\end{document}